\documentclass[a4paper,12pt]{amsart}

\usepackage{amsmath,amssymb,txfonts}
\usepackage[alphabetic]{amsrefs}
\usepackage{geometry}
\usepackage{exscale}
\usepackage[new]{old-arrows}

\newtheorem{thm}{Theorem}[section]
\newtheorem{cor}[thm]{Corollary}
\newtheorem{defn}[thm]{Definition}
\newtheorem{rem}[thm]{Remark}
\numberwithin{equation}{section}

\begin{document}

\title[Verifying spiral minimal product structure through the Takahashi Theorem]{Verifying spiral minimal product structure\\ through the Takahashi Theorem}
\author{Yongsheng Zhang}
\address{Academy for Multidisciplinary Studies, Capital Normal University, Beijing 100048, P. R. China}
\email{yongsheng.chang@gmail.com}
\date{2026/04/17}


\begin{abstract}
     This note expands on a remark of \cite{LZ}. Namely, we will verify the spiral minimal product structure through the Takahashi Theorem with full computational details which were omitted in \cite{LZ}.
\end{abstract}

\maketitle

\section{Introduction}

        Minimal varieties in spheres are in one-to-one correspondence with minimal cones in Euclidean spaces.
        As prototypes of minimal models at an infinitesimal scale,
          the study of minimal cones is of particular importance.

        Recently,
          the author was interested in exploring structures for minimal submanifolds in spheres and related topics.
        In \cite{TZ} a proof of constant minimal product of minimal submanifolds is given with a complete description of the second fundamental form and combining Lawlor's curvature criterion and the theory of isoparametric foliations of spheres shows that minimal products among minimal isoparametric hypersurfaces and focal submanifolds can span area-minimizing cones whenever the products have dimension greater or equal to $36$.

        Given minimal submanifolds $\{L_i^{k_i}\}$ in unit Euclidean spheres $\mathbb S^{N_i}$ respectively for $i=1,\cdots, n$,
          their constant minimal product (up to $\pm$ signs) is in the following
          \begin{equation}\label{prod}
            L_1\times\cdots
            \times L_n
            \triangleq
            \left(\lambda_1 L_1, \cdots, \lambda_n  L_n\right)\subset \mathbb S^{N_1+\cdots +N_n+n-1}
          \end{equation}
          where $\lambda_i=\sqrt{\frac{k_i}{k}}$ and $k=\sum_{i=1}^n k_i>0$ provides a nice algorithm to generate new minimal submanifolds in spheres due to the property that the structure only depends on dimensions $\{k_i\}$ regardless of concrete figures of $\{L_i\}$ at all,
          see \eqref{dotsss} below.
        This result should be classical and the author found a very quick proof in \cite{X} through the Takahashi Theorem around the end of 2020.

        Inspired by \cite{CLU},
          we discovered in \cite{LZ} a new algorithm beyond \eqref{prod} to produce minimal submanifolds in spheres.
        Instead of using constant factors in \eqref{prod},
          we employ components of suitable immersed submanifolds in certain sphere.
        For simplicity,
          let us focus on the situation with merely two inputs in the algorithm,
          namely,
          \begin{equation}
            \left(\gamma_1 L_1,  \gamma_2  L_2\right)\subset \mathbb S^{N_1+N_2+1}
          \end{equation}
          where $\gamma=(\gamma_1(t), \gamma_2(t))$ is an immersed curve in $\mathbb S^3\subset \mathbb{C}\oplus \mathbb{C}$.

        For technical reason,
          we rely on the concept of $C$-totally real submanifold $L$,
          i.e.,
          $Jx\perp T_xL$ for every point $x\in L\subset \mathbb S^{2n+1}$ where $J$ is the standard complex structure of ambient complex Euclidean space.
        Every submanifold $M\subset \mathbb S^{n}$ becomes $C$-totally real through the trivial {extension} $M\times \{0\}\subset {\mathbb S^{n}\times\{0\}\subset \mathbb S^{2n+1}}\subset \mathbb R^{n+1}\oplus i\cdot \mathbb R^{n+1}$.

        We shall focus on the situation that $L_1^{k_1}\subset \mathbb S^{2n_1+1}$ and $L_2^{k_2}\subset \mathbb S^{2n_2+1}$ are two $C$-totally real submanifolds.
        Their spiral product by $\gamma$ is
          \begin{eqnarray}
            \ \ \ \ \ \  \ \ \  G_\gamma: \mathbb R\times L_1\times L_2 \longrightarrow \mathbb S^{2n_1+2n_2+3}\ \ by \ \
            \big(t,x,y\big)\longmapsto \Big(\gamma_1(t) \, x,\, \gamma_2(t)\, y\Big) .
            \label{expr}
          \end{eqnarray}
        When $G_\gamma$ is a minimal immersion,
          we call it a spiral minimal product.
        The questions when the spiral minimal product induces a compact immersed or embedded minimal submanifold and what type the submanifold can be were studied in \cite{LZ}.

\section{Spiral minimal product}

        Let $\gamma_1=a e^{is_1}$ and $\gamma_2=b e^{is_2}$ where $a, b, s_1, s_2$ are $C^2$ functions in $t$.
        \begin{defn}
        If $a(t)$ and $b(t)$ remain unchanged,
          $G_\gamma$ is called of steady magnitudes,
          and otherwise of varying magnitudes.
        \end{defn}
        Suppose $L_1$ and $L_2$ are $C$-totally real minimal submanifolds,
          we want to construct their spiral minimal product for some $\gamma$ in \eqref{expr}.
        In the case of varying magnitudes,
          with the parameter $s$ such that $a(s)=\cos s$ and $b(s)=\sin s$,
          local pieces of $\gamma$ are given by
          \begin{equation}\label{e1}
            b^2\dot s_2
            =C_1 {a^2\dot s_1},
          \end{equation}
          \begin{equation}\label{dotG}
            -k_1\frac{b}{a}+k_2 \frac{a}{b}
            -
            \frac{ab\left[
            (\dot s_1)^2-(\dot s_2)^2
            \right]}{(a\dot s_1)^2+(b\dot s_2)^2}
            =-
            \frac{\dot \Theta}{2\Theta  \left(
            1+\Theta
            \right)}
          \end{equation}
          where $C_1\in \mathbb{R}$ is a parameter and $\Theta=(a\dot s_1)^2+(b \dot s_2)^2$.
        Here the pair of \eqref{e1} and \eqref{dotG} is exactly the condition for the second fundamental form of $G_\gamma$ having vanishing trace with respect to all normal directions,
          i.e.,
          $G_\gamma$ to be minimal.

        Plug relation \eqref{e1} into \eqref{dotG} for
          \begin{equation}\label{dotG2}
            -k_1\frac{b}{a}+k_2 \frac{a}{b}
            -
            \frac{b^4-C_1^2 a^4}{ab(b^2+C_1^2a^2)}
            =-
            \frac{\dot \Theta}{2\Theta  \left(
            1+\Theta
            \right)}.
          \end{equation}
        And we can solve this by
          \begin{equation}\label{Gexp}
            \displaystyle
            \Theta
            =\dfrac
            {
            {
            1+\left(C_1^2-1\right)\cos^2 s
            }
            }
            { C_2
            \left(\cos s\right)^{2k_1+2}
            \left( \sin s\right)^{2k_2+2}
            \,
            -
            {
            1-\left(C_1^2-1\right)\cos^2 s
            }
            }
          \end{equation}
          where $C_2$ is another parameter to make the denominator of \eqref{Gexp} positive on some nonempty set $J\subset (0,\frac{\pi}{2})$.\footnote{
          As explained in the superseding version of \cite{LZ},
          these constants have a Hamiltonian interpretation:
          $C_1$ is the ratio of the two conserved Noether momenta,
          while $C_2$ is the normalized energy level of the reduced Hamiltonian system.}
          It turns out that the largest $J$ is always a connected interval and over $J$ the local solution curve is given by
          \begin{equation}\label{dotsss}
            \begin{pmatrix}
            \dot s_1\\
            \dot s_2
            \end{pmatrix}
            =
            \pm
            \sqrt{\frac{1}{
            {C_2
            \left(\cos s\right)^{2k_1+2}
            \left( \sin s\right)^{2k_2+2}
            \,
            -
            {
            1-\left(C_1^2-1\right)\cos^2 s
            }
            }
            }
            }
            \,
            \begin{pmatrix}
            \tan s\\
            C_1\cot s
            \end{pmatrix}.
          \end{equation}
        It should be pointed out here that when $C_2\downarrow \min \left\{ \frac{1+\left(C_1^2-1\right)\cos^2 s}{\left(\cos s\right)^{2k_1+2} \left( \sin s\right)^{2k_2+2}} \Big| s\in\left(0,\frac{\pi}{2}\right) \right\} $,
          the spiral minimal products of varying magnitudes ($a, b$) limit to a spiral minimal product of steady magnitudes.

        For fixed admissible $C_1, C_2$ in the local solution curve,
          one can assemble $+$ and $-$ parts of \eqref{dotsss} alternately for a globally defined solution curve $\gamma$ in a $C^1$-manner.
        Furthermore,
          from Harvey-Lawson's extension result for minimal submanifolds with $C^1$-joints we know that the global solution curve $\gamma$ is analytic.

        This way,
          we derive an algorithm of spiral minimal products.
        It can be observed (see \cite{LZ}) that the resulting product is again $C$-totally real if and only if $C_1=-1$.
        So,
          one can apply the construction multiple times
          \begin{equation}\label{mprod}
            \Big ( \big(( L_1\times_{\beta_1} L_2) \times_{\beta_2} L_3\big)\cdots
            \Big)
            \times_{\beta_{n-1}} L_n
          \end{equation}
          where each $\beta_i$ is a solution curve with $C_1=-1$.
        Certain global properties of solution curves with $C_1=-1$ have been studied in \cite{LZ} using the theory of calibrations.
        To be more precise,
          every solution curve $\gamma$ corresponding to $C_1=-1$ is either an embedded closed curve in $\mathbb S^3$ or an immersed $\mathbb R$ in $\mathbb S^3$ with no self-intersection.
        If this were not true,
          then $\gamma$ would have either a transverse self-intersection
          or a tangential self-intersection with opposite tangent directions
          (the latter is impossible by uniqueness of solutions).
        As solution curves by \eqref{dotsss} are independent of figures of the inputs,
          one can use $\gamma$ with inputs $\mathbb S^{k_i}\times \{0\}\subset \mathbb S^{2k_i+1}$ for $i=1,2$ to generate a special Legendrian immersion whose Legendrian angle is a constant as $\mathbb{R}\times \mathbb S^{k_1}\times \mathbb S^{k_2}$ is connected.
        Hence,
          take a cone over it and one gets a special Lagrangian cone with self-intersection in codimension one.
        This is impossible by applying Harvey-Lawson's decomposition result on the special Lagrangian calibration,
          or alternatively by the Almgren's big regularity theorem in geometric measure theory.

        It should also be emphasized that the spiral minimal product structure with $C_1=-1$ is very special since it can be repeatedly applied to special Legendrian inputs to manufacture new special Legendrian varieties.
        The cones spanned over special Legendrian varieties are special Lagrangian cones which play important roles in many different branches of geometry.
        If we combine this structure with Hopf fibration and its reverse horizontal lifting,
          a generalized Delaunay construction for the category of minimal Lagrangians in complex projective spaces can be established.

        Since every minimal submanifold in sphere becomes $C$-totally real minimal in some higher dimensional sphere,
          for example through the complexification of ambient Euclidean space,
          the algorithm \eqref{mprod} together with the complexification procedure can be applied repeatedly for any given finitely many minimal submanifolds in spheres without the restriction $C_1=-1$ in each step.

        In summary,
          spiral minimal products provide a new algorithm to produce new examples of minimal submanifolds in spheres and new examples of special Lagrangian cones in complex Euclidean spaces.

\section{Verifying spiral minimal product through the Takahashi Theorem}

        Comparing with the original proof of the constant minimal product structure mentioned in \S 1,
          at the end of \cite{LZ} we describe how to understand the spiral minimal product structure from a different angle.
        Due to length considerations,
          we omitted computational details in \cite{LZ} and instead mentioned that ``one can check through some basic but delicate computations" by the Takahashi Theorem and listed several key ingredients therein.

        In this section,
          full details of these delicate computations will be given.
        Let {$\Delta_{g}$} be the standard Laplace-Beltrami operator of Riemannian $(M^k, g)$ for $C^\infty(M)$.
        Then it automatically defines an operator for $C^2$-vector-value map $f:M\longrightarrow \mathbb{R}^{n+1}$.
        If we use local coordinates $(x_1,\cdots, x_k)$,
          then one can derive the well-known formula
          \begin{equation}
            \Delta_gf=\frac{1}{\sqrt g} \frac{\partial}{\partial x_i}\left(\sqrt g g^{ij} \frac{\partial f}{\partial x_j}\right)
          \end{equation}
          where $g=\det(g_{ij})$ and $(g^{ij})$ is the inverse of the matrix $(g_{ij})$.

        The celebrated Takahashi Theorem captures a very beautiful picture of $f$ deciding a minimal immersion into some Euclidean sphere in terms of this Laplace operator.

        \begin{thm}[\cite{T}]\label{T}
        Let $(M, g)$ be a $k$-dimensional Riemannian manifold and $f:M\longrightarrow \mathbb{R}^{n+1}$ an isometric immersion.
        If there exists $\lambda\neq 0$ such that
          $
          \Delta_gf=-\lambda f,
          $
          then it follows that (1) $\lambda>0$;
          (2) $ f(M)\subset S^n(r)$ where $r^2=\frac{k}{\lambda}$;
          (3) $f:M\longrightarrow S^n(r)$ is a minimal immersion.
        Conversely,
          if $f$ defines a minimal immersion into $\mathbb S^{n}(r)$,
          then
          $
          \Delta_gf=-\lambda f
          $
          and $\lambda =\frac{k}{r^2}$.
        \end{thm}

        Note that the spiral minimal product structure in fact works for immersion situation.
        So,
          in the sequel,
          we will assume that $f_1:L_1\longrightarrow \mathbb S^{2n_1+1}$ and $f_2:L_2\longrightarrow \mathbb S^{2n_2+1}$ are two given $C$-totally real minimal immersions.
        Then a spiral product $G_\gamma$ of them induces a metric $g_\gamma$ on $\mathbb{R}\times L_1\times L_2$ by
          \begin{equation}\label{metric}
            g_\gamma
            =
            dt_{arc}^2
            \oplus
            a^2(t_{arc})\, g_1\oplus b^2(t_{arc})\, g_2
          \end{equation}
          where $t_{arc}$ is an arc-parameter of solution curve $\gamma\subset \mathbb S^3$.

        According to Theorem \ref{T},
          we have
          $$\Delta_{g_1}f_1=-k_1f_1
            \text{\ \ \ \  and \ \ \ \ }
            \Delta_{g_2}f_2=-k_2f_2.
          $$
        And the remark in \cite{LZ}
          \footnote{See also Remark~1.7 in its superseding version.}
          states the following,
          which we formulate a theorem here.
        \begin{thm}\label{main}
        Given two $C$-totally real minimal immersed submanifolds
          $$f_1:L_1\longrightarrow \mathbb S^{2n_1+1} \text{\ \ \ \  and \ \ \ \ } f_2:L_2\longrightarrow \mathbb S^{2n_2+1}.$$
        Suppose $G_\gamma$ is given by solving \eqref{e1} and \eqref{dotG}.
        Then it follows that
          \begin{equation}\label{GgEq}
            {\Delta_{g_\gamma}}
            G_\gamma
            =-(k_1+k_2+1)G_\gamma\, .
          \end{equation}
        \end{thm}
        \begin{rem}
        The statement is true immediately due to Theorem \ref{T} and the geometric meaning of solution curve $\gamma$ of the paired system \eqref{e1} and \eqref{dotG}.
        Here we directly check that \eqref{GgEq} holds for all solutions of the paired system.
        \end{rem}
        \begin{proof}
        Based on \eqref{metric} we know that
          \begin{eqnarray}
            && {\Delta_{g_\gamma}}
            G_\gamma
            \\
            & =&
            \left(  \frac{\gamma_1}
            {a^2}
            \Delta_{g_1}
            f_1
            \, ,\,
            \frac{\gamma_2}
            {b^2}
            \Delta_{g_2}
            f_2
            \right)
            +
            \left(\gamma''_1 f_1\, ,\, \gamma''_2 f_2\right)
            +
            \left(k_1\frac{a'}{a}+k_2\frac{b'}{b}\right)
            \left( \gamma'_1 f_1\, ,\, \gamma'_2 f_2\right)
            \nonumber         \\
            &=&
            \left(
            \left[- \frac{k_1}
            {a^2}
            +
            \frac{\gamma''_1}{\gamma_1}
            +
            k_1\frac{a'\gamma_1'}{a\gamma_1}
            +
            {k_2\frac{b'\gamma_1'}{b\gamma_1}}
            \right]
            \gamma_1f_1
            \, ,\,
            \left[- \frac{k_2}
            {b^2}
            +
            \frac{\gamma''_2}{\gamma_2}
            +
            {k_1\frac{a'\gamma'_2}{a\gamma_2}}
            +
            k_2\frac{b'\gamma_2'}{b\gamma_2}
            \right]
            \gamma_2f_2
            \right)
            \nonumber
          \end{eqnarray}
          where derivative $'$ is taken with respect to an arc-parameter $t_{arc}$ of $\gamma$.

        Since the expressions of complex coefficients of $\gamma_1f_1$ and $\gamma_2f_2$ are completely symmetric about indices (and $a,b$),
          we shall only focus on what happens to the coefficient of $\gamma_1f_1$,
          namely,
          \begin{equation}\label{coef1}
            - \frac{k_1}
            {a^2}
            +
            \frac{\gamma''_1}{\gamma_1}
            +
            k_1\frac{a'\gamma_1'}{a\gamma_1}
            +
            {k_2\frac{b'\gamma_1'}{b\gamma_1}}.
          \end{equation}
        The relevant terms are expressed as follows
          \begin{eqnarray}
            \gamma_1'
            &=&
            [a'+ias_1']e^{is_1}
            \nonumber\\
            \gamma_1''
            &=&
            [a''-a(s_1')^2+i(2a's_1'+as_1'')]e^{is_1}
            \nonumber        \\
            \gamma_2'
            &=&
            [b'+ibs_2']e^{is_2}
            \nonumber        \\
            \gamma_2''
            &=&
            [b''-b(s_2')^2+i(2b's_2'+bs_2'')]e^{is_2}.
            \nonumber
          \end{eqnarray}

        Splitting Expression \eqref{coef1} into its real and imaginary parts yields
          \begin{eqnarray}
            &&
            \label{coef2}
            \  \  - \frac{k_1}
            {a^2}
            +
            \frac{\gamma''_1}{\gamma_1}
            +
            k_1\frac{a'\gamma_1'}{a\gamma_1}
            +
            {k_2\frac{b'\gamma_1'}{b\gamma_1}}
            \\
            &=&
            \left[
            - \frac{k_1}
            {a^2}
            +
            \frac{a''-a(s_1')^2}{a}
            +k_1\frac{(a')^2}{a^2}
            +k_2\frac{a'b'}{ab}
            \right]
            +
            i \left[
            \frac{2a's_1'+as_1''}{a}
            +k_1
            \frac{a's_1'}{a}
            +
            k_2
            \frac{b's_1'}{b}
            \right].
            \nonumber
          \end{eqnarray}
        Recall that
          $
          \frac{ds}{d t_{arc}}=\frac{1}{\sqrt{1+\Theta}}
          $
          and the basic relations
          \begin{equation}\label{tarc}
            \frac{d}{d t_{arc}} s_1=\frac{\dot s_1}{\sqrt{1+\Theta}}
            =\frac{\tan s}{\sqrt{C_2
            \left(\cos s\right)^{2k_1+2}
            \left( \sin s\right)^{2k_2+2}}}
            =\frac{1}{\sqrt{C_2}\,\, a^{k_1+2}b^{k_2}}\, ,
          \end{equation}
          \begin{equation}
            \frac{d}{d t_{arc}} s_2
            =\frac{\dot s_2}{\sqrt{1+\Theta}}
            =\frac{C_1\cot s}{{\sqrt{C_2
            \left(\cos s\right)^{2k_1+2}
            \left( \sin s\right)^{2k_2+2}}}}
            =\frac{C_1}{\sqrt{C_2}\, \, a^{k_1} b^{k_2+2}} ,
          \end{equation}
          \begin{equation}\label{tarc2}
            \frac{d}{d t_{arc}} a
            =-\frac{b}{\sqrt{1+\Theta}}
            =
            -\frac
            {
            \sqrt{C_2
            a^{2k_1+2}
            b^{2k_2+2}
            \,
            -
            {
            1-(C_1^2-1) a^2
            }
            }
            }
            {\sqrt{C_2}\,
            a^{k_1+1}
            b^{k_2}} ,
          \end{equation}
          \begin{equation}\label{tarc3}
            \,      \frac{d}{d t_{arc}} b=\frac{a}{\sqrt{1+\Theta}}
            =
            \frac
            {
            \sqrt{C_2
            a^{2k_1+2}
            b^{2k_2+2}
            \,
            -
            {
            1-(C_1^2-1) a^2
            }
            }
            }
            {\sqrt{C_2}\,
            a^{k_1}
            b^{k_2+1}}
            .
          \end{equation}

        Since $s_1'=\frac{1}{\sqrt{C_2}\,\, a^{k_1+2}b^{k_2}}$,
          we have by \eqref{tarc} to \eqref{tarc3} that
          \begin{eqnarray*}
            \ \ \ \ \ \ \ s_1''
            =
            -\frac{\left[(k_1+2)\frac{a'}{a}+k_2\frac{b'}{b}\right]a^{k_1+2}b^{k_2}
            }{\sqrt{C_2}(a^{k_1+2}b^{k_2})^2}
            =
            -\frac{-\frac{(k_1+2)b}{a\sqrt{1+\Theta}}+\frac{k_2a}{b\sqrt{1+\Theta}}
            }{\sqrt{C_2}    a^{k_1+2}b^{k_2}},
          \end{eqnarray*}
          $$
            \, \frac{a's_1'}{a}=
            \frac{-\frac{b}{a\sqrt{1+\Theta}}
            }{\sqrt{C_2}    a^{k_1+2}b^{k_2}},\
          $$
          $$
            \frac{b's_1'}{b}
            =
            \frac{\frac{a}{b\sqrt{1+\Theta}}
            }{\sqrt{C_2}    a^{k_1+2}b^{k_2}}.
          $$
        Hence,
          the imaginary part in \eqref{coef2} vanishes.

        To compute the real part of \eqref{coef2} which we hope to be $-k_1-k_2-1$,
          we have to deal with more subtle computations.
        The easy terms are
          \begin{eqnarray*}
            k_1\frac{(a')^2}{a^2}
            &=
            &
            \ \ \   \frac{k_1b^2}{a^2(1+\Theta)},
            \\
            k_2\frac{a'b'}{ab}
            &=
            &
            \ \ \        -\frac{k_2}{1+\Theta},
            \\
            (s_1')^2
            &=&
            \frac{1}{{C_2}\,\, a^{2k_1+4}b^{2k_2}}.
          \end{eqnarray*}
        According to \eqref{tarc2} and \eqref{tarc3},
          denote $IV = k_1\frac{(a')^2}{a^2}+k_2\frac{a'b'}{ab} =\frac{k_1b^2-k_2a^2}{a^2(1+\Theta)} $.

        Since $a'=-\frac{b}{\sqrt{1+\Theta}}$,
          we have
          \begin{eqnarray*}
            \frac{a''}{a}
            &=&
            -\frac{b'}{a\sqrt{1+\Theta}}
            +
            \frac{b\Theta'}{2a(1+\Theta)^{\frac{3}{2}}}
            \\
            &=&
            -\frac{b'}{a\sqrt{1+\Theta}}
            +
            \frac{b\dot \Theta}{2a(1+\Theta)^{2}}
            \\
            &=&
            -\frac{1}{{1+\Theta}}
            +
            \frac{b\dot \Theta}{2a(1+\Theta)^{2}}
            \\
            &\triangleq &
            \ \ \ \  \  I \ \ \ \ \ \ + \ \ \ \ \ \ \ \ X\ \ \ \  .
          \end{eqnarray*}
        Note that
          \begin{eqnarray*}
            X=     \frac{b\dot \Theta}{2a(1+\Theta)^{2}}
            & =&
            \frac{b\Theta}{a(1+\Theta)}
            \frac{\dot \Theta}{2\Theta(1+\Theta)}
            \\
            \text{by }\eqref{dotG2}         \ \ \ \         \ \ \ \ \ \ \ \ \        &=&
            \frac{b(1+(C_1^2-1)a^2)}{aC_2  a^{2k_1+2} b^{2k_2+2}}
            \left(
            k_1 \frac{b}{a}
            -k_2\frac{a}{b}
            +\frac{b^4-C_1^2a^4}{ab(b^2+C_1^2a^2)}
            \right)
            \\
            &=&
            \frac{b(1+(C_1^2-1)a^2)}{aC_2  a^{2k_1+2} b^{2k_2+2}}
            \left(
            k_1 \frac{b}{a}
            -k_2\frac{a}{b}
            \right)
            +
            \frac{b^4-C_1^2a^4}{a^2C_2  a^{2k_1+2} b^{2k_2+2}}
            \\
            &\triangleq &
            \ \ \ \ \ \ \   II\ \ \ \ \ \  \ \ \ + \ \ \ \ \ \ \ \ III\ \ \ \ \ \ \ .
          \end{eqnarray*}
        Here we need to check that \eqref{dotG2} makes sense for spiral minimal product of varying magnitudes,
          i.e.,
          $a'b'$ not always zero.
        The limiting case $-$ spiral minimal product of steady magnitudes will be discussed later.

        As
          $$
            \frac{1}{1+\Theta}
            =\frac{C_2 a^{2k_1+2}b^{2k_2+2}-1-(C_1^2-1)a^2}{C_2 a^{2k_1+2}b^{2k_2+2}}
            =1-\frac{1+(C_1^2-1)a^2}{C_2 a^{2k_1+2}b^{2k_2+2}},
          $$
          we get
          \begin{eqnarray*}
            I+III
            &=&
            -\frac{1}{1+\Theta}
            +  \frac{b^4-C_1^2a^4}{a^2C_2  a^{2k_1+2} b^{2k_2+2}}\\
            &=&
            -1+\frac{a^2(1+(C_1^2-1)a^2)+b^4-C_1^2a^4}{a^2C_2  a^{2k_1+2} b^{2k_2+2}}
            \\
            &=&
            -1+\frac{a^2-a^4+b^4}{a^2C_2  a^{2k_1+2} b^{2k_2+2}}
            \\
            &=&
            -1+\frac{b^2}{a^2C_2  a^{2k_1+2} b^{2k_2+2}}
            \\
            &=&
            -1+\frac{1}{C_2  a^{2k_1+4} b^{2k_2}}.
          \end{eqnarray*}

        Now let us compute
          \begin{eqnarray*}
            IV+II
            &=&
            \frac{k_1 b^2- k_2 a^2}{a^2(1+\Theta)}
            +
            \frac{b(1+(C_1^2-1)a^2)}{aC_2  a^{2k_1+2} b^{2k_2+2}}
            \left(
            k_1 \frac{b}{a}
            -k_2\frac{a}{b}
            \right)
            \\
            &   =&
            \frac{k_1 b^2- k_2 a^2}{a^2}
            \\
            &=&
            \frac{k_1b^2}{a^2}-k_2.
          \end{eqnarray*}

        As a result,
          the real part of \eqref{coef2}
          \begin{eqnarray*}
            &&
            - \frac{k_1}
            {a^2}
            +
            \frac{a''-a(s_1')^2}{a}
            +k_1\frac{(a')^2}{a^2}
            +k_2\frac{a'b'}{ab}
            \\
            &=&
            - \frac{k_1}
            {a^2}
            + I+X -(s_1')^2 +IV
            \\
            &=&
            - \frac{k_1}
            {a^2}
            + I+II+III -(s_1')^2 +IV
            \\
            &=&
            - \frac{k_1}
            {a^2}
            + I+III -(s_1')^2 +II+IV
            \\
            &=&
            - \frac{k_1}
            {a^2}
            -1+\frac{1}{C_2  a^{2k_1+4} b^{2k_2}} -(s_1')^2 + \frac{k_1b^2}{a^2}-k_2
            \\
            &=&
            - \frac{k_1}
            {a^2}
            -1 + \frac{k_1b^2}{a^2}-k_2
            \\
            &=&
            -k_1-k_2-1.
          \end{eqnarray*}
        This concludes the proof for varying magnitudes.

        Let us consider the computations for steady magnitudes,
          which means $a$ and $b$ are nonzero constants $\cos s_{C_1}$ and $\sin s_{C_1}$ respectively.
        Clearly,
          $s_1''=0$ by \eqref{tarc} and the imaginary part of \eqref{coef2} vanishes.
        The real part gets simpler as
          $ - \frac{k_1}
          {a^2}
          -(s_1')^2$.
        By \eqref{tarc2} we have $C_2a^{2k_1+2}b^{2k_2+2}=1+(C_1^2-1)a^2$ for this situation.
        Together with \eqref{tarc},
          we get
          $$(s_1')^2=   \frac{1}{{C_2}\,\, a^{2k_1+4}b^{2k_2}}=\frac{\frac{b^2}{a^2}}{1+(C_1^2-1)a^2}.$$
        In \cite{LZ} it is computed that
          \begin{equation}   \label{sCd}
            C_1^2-1
            =\dfrac
            {-(2k_1+2)\tan s_{C_1}+(2k_2+2)\cot s_{C_1}}
            {\, 2k_1\tan s_{C_1}-(2k_2+2)\cot s_{C_1}  \, }
            \dfrac{1}{\cos^2 s_{C_1}}
            \, .
          \end{equation}
        By virtue of \eqref{sCd},
          we arrive at
          $$(s_1')^2=   \frac{\tan^2 s_{C_1}              \left(
            2k_1\tan s_{C_1}-(2k_2+2)\cot s_{C_1}
            \right)      }{-2\tan s_{C_1}}
            =
            -k_1\tan^2 s_{C_1}
            +
          k_2+1.$$
        Therefore,
          $ - \frac{k_1}
          {a^2}-(s_1')^2
          =
          -k_1-k_2-1
          $
          and we finish the proof.
        \end{proof}

        A corollary is the following.

        \begin{cor}\label{cor33}
        Let $L_1, L_2$ be given as in Theorem \ref{main} of dimension $k_1, k_2>0$.
        Then immersed submanifolds of the type $G_\gamma(t, x, y)=\big(\gamma_1f_1(x), \gamma_2f_2(y)\big)$ satisfying
          $$
            {\Delta_{g_\gamma}}
            G_\gamma
            =-(k_1+k_2+1)G_\gamma\,
          $$
          by immersed curves $\gamma=(\gamma_1,\gamma_2)\subset\mathbb S^3$ with $\gamma_1\gamma_2\neq 0$ are exactly the spiral minimal products of $f_1$ and $f_2$ with varying or steady magnitudes.
        \end{cor}
        \begin{rem}
        Under the dimension assumption of Corollary \ref{cor33},
          $\gamma_1\gamma_2\neq 0$ is necessary to obtain an immersion $G_\gamma(t, x, y)$.
        \end{rem}

        \textbf{Acknowledgments.}
          The author would like to thank the referee for careful reading and kind suggestions.
        This work is supported by the NSFC under Grant Numbers 12022109 and 11971352.

\end{document}